\documentclass[11pt, final]{amsart}

\title[Ricci soliton hypersurfaces in the complex hyperbolic spaces]{Homogeneous Ricci soliton hypersurfaces in the complex hyperbolic spaces}
\author[T.\ Hashinaga]{Takahiro Hashinaga}
\author[A.\ Kubo]{Akira Kubo}
\author[H.\ Tamaru]{Hiroshi Tamaru}
\address[T.\ Hashinaga]{Department of Mathematics, Hiroshima University, 
Higashi-Hiroshima 739-8526, Japan}
\address[A.\ Kubo]{Department of Mathematics, Hiroshima University, 
Higashi-Hiroshima 739-8526, Japan}
\address[H.\ Tamaru]{Department of Mathematics, Hiroshima University, 
Higashi-Hiroshima 739-8526, Japan}
\email[A.\ Kubo]{akira-kubo@hiroshima-u.ac.jp}
\email[T.\ Hashinaga]{hashinaga@hiroshima-u.ac.jp}
\email[H.\ Tamaru]{tamaru@math.sci.hiroshima-u.ac.jp}

\thanks{The first author was supported in part by Grant-in-Aid for JSPS Fellows (11J05284).
	The third author was supported in part by KAKENHI (24654012).}

\usepackage{a4}
\usepackage{amssymb}
\usepackage{enumerate}

\theoremstyle{definition}
\newtheorem{Def}{Definition}[section]

\theoremstyle{plain}
\newtheorem{Prop}[Def]{Proposition}
\newtheorem{Thm}[Def]{Theorem}
\newtheorem{Lem}[Def]{Lemma}

\theoremstyle{definition}
\newtheorem{Remark}[Def]{Remark}

\newtheorem*{Theorem}{Theorem}

\numberwithin{equation}{section}

\def\fr#1{\mathfrak{#1}}

\newcommand{\ad}{\mathop{\mathrm{ad}}\nolimits}

\newcommand{\Der}{\mathop{\mathrm{Der}}\nolimits}
\renewcommand{\span}{\mathop{\mathrm{span}}\nolimits}

\newcommand{\CH}{\mathbb{C} \mathrm{H}}
\newcommand{\RH}{\mathbb{R} \mathrm{H}}
\newcommand{\Ric}{\mathop{\mathrm{Ric}}\nolimits}
\newcommand{\id}{\mathop{\mathrm{id}}\nolimits}

\newcommand{\iro}{}
\newcommand{\irotwo}{}

\begin{document}

\begin{abstract}
A Lie hypersurface in the complex hyperbolic space 
is an orbit of a cohomogeneity one action without singular orbit. 
In this paper, we 
classify 
Ricci soliton Lie hypersurfaces in the complex hyperbolic spaces. 
\end{abstract}

\maketitle 

\section{Introduction}
Homogeneous submanifolds in Riemannian symmetric spaces of noncompact type $G/K$ 
have provided a lot of interesting examples of submanifolds. 
We refer to \cite{B, BDT, BT03, BT04, BT07, DDK, HHT, KT, T} 
and references therein. %
Typical and the simplest examples are given by the nilpotent part $N$ 
of the Iwasawa decomposition $G = KAN$. 
All orbits of $N$ in $G/K$ are isometrically congruent to each other 
(see \cite{BDT, KT}), 
and they are known to be Ricci solitons (\cite{L01}). 
Recall that a 
complete Riemannian manifold $(M, g)$ is called a \textit{Ricci soliton} if 
\begin{equation}
	\Ric_g = cg - \frac{1}{2}\mathcal{L}_X g
\end{equation}
holds
for some $c \in \mathbb{R}$ and some complete vector field $X$ on $M$, 
where $\Ric_g$ denotes the Ricci operator of $(M, g)$ and 
$\mathcal{L}_X$ is the usual Lie derivative. 
The vector field $X$ is called the 
\textit{potential vector field}. 
A Ricci soliton is a natural generalization of an Einstein manifold. 

In symmetric spaces of noncompact type,
homogeneous hypersurfaces called Lie hypersurfaces are of particular interest. 
For an isometric action on a Riemannian manifold, maximal
dimensional orbits are said to be \textit{regular}, 
and other orbits \textit{singular}. 
The codimension of a regular orbit is called the \textit{cohomogeneity} of an action. 
A \textit{Lie hypersurface} is 
an orbit of a cohomogeneity one action without singular orbit. 
The notion of Lie hypersurfaces has been introduced by Berndt (\cite{B}). 
For symmetric spaces of noncompact type $G/K$, 
the orbits of a cohomogeneity one action without singular orbit 
form a Riemannian foliation (\cite{BB}). 
Furthermore, when $G/K$ is irreducible, 
such actions have been completely classified up to orbit equivalence (\cite{BT03}). 
Recall that two isometric actions on a Riemannian manifold are said to be 
\textit{orbit equivalent} if 
there exists an isometry of the manifold mapping the orbits 
of one of these actions onto the orbits of the other action. 

In this paper, we focus on Lie hypersurfaces in the 
complex hyperbolic spaces 
\begin{equation}
\CH^n = \mathrm{SU}(1,n)/\mathrm{S}(\mathrm{U}(1)\times\mathrm{U}(n)) 
\end{equation}
with $n \geq 2$. 
Owing to \cite{BT03}, 
for rank one symmetric spaces of noncompact type, 
there exist exactly two cohomogeneity one actions 
without singular orbit up to orbit equivalence. 
One of the actions is given by $N$, 
the nilpotent part of the Iwasawa decomposition of the isometry group. 
The orbits of $N$, which are isometrically congruent to each other, are horospheres. 
The other action induces the so-called 
\textit{solvable foliation}.
In the case of the real hyperbolic space 
$\RH^n$, 
intrinsic geometry of the Lie hypersurfaces is well-known. 
In fact, the horosphere is flat, and 
the solvable foliation consists of 
a totally geodesic $\RH^{n-1}$ and its equidistant hypersurfaces. 
It is easy to see that 
all Lie hypersurfaces in $\RH^n$ have constant curvature. 
On the contrary, in the case of $\CH^n$ with $n \ge 2$, 
the situation is much more interesting and nontrivial. 
In this case, the solvable foliation consists of
the homogeneous ruled minimal hypersurface and its equidistant hypersurfaces. 
The homogeneous ruled minimal hypersurface in $\CH^n$ 
is the ruled real hypersurface determined by 
a horocycle in a totally geodesic $\RH^2$ in $\CH^n$ (\cite{LR}). 
Note that this is isometric to a fan (we refer to \cite{GG}).

The purpose of this paper is to study and classify Ricci soliton Lie hypersurfaces in $\CH^n$. 
It is well-known that there exist no Einstein hypersurfaces in 
$\mathbb{C} \mathrm{H}^n$ (see \cite{R}). 
On the other hand, as we mentioned before, 
the horosphere is a Ricci soliton. 
Our main theorem classifies Ricci soliton Lie hypersurfaces in $\CH^n$. 
It would be interesting that the result is different in the cases $n=2$ and $n>2$. 

\begin{Theorem}
A Lie hypersurface in $\mathbb{C} \mathrm{H}^n$ is a Ricci soliton if and only if 
\begin{enumerate}
\item[$(1)$]
it is isometrically congruent to a horosphere, or 
\item[$(2)$]
$n=2$ and it is isometrically congruent to 
the homogeneous ruled minimal hypersurface. 
\end{enumerate}
\end{Theorem} 

The above theorem is relevant to the results by Cho and Kimura (\cite{CK11}). 
Among others, for real hypersurfaces in non-flat complex space forms, 
they proved 
\begin{enumerate}
\item[-]
there do not exist compact Hopf hypersurfaces which are Ricci solitons, 
\item[-]
there do not exist ruled hypersurfaces which are gradient Ricci solitons. 
\end{enumerate}
Recall that a Ricci soliton is said to be 
\textit{gradient} if the potential vector field 
can be expressed as the gradient of a smooth function. 
We also recall that 
a hypersurface in $\CH^n$ is said to be \textit{Hopf} 
if the structure vector field $J \xi$ is an eigenvector of the shape operator, 
where $J$ is the complex structure and $\xi$ is a unit normal vector. 
Our theorem implies that the assumptions on these 
nonexistence results cannot be removed. 
Note that a horosphere is a Hopf hypersurface (see \cite{B, HHT}), 
which is a Ricci soliton, but not compact. 
When $n=2$, the homogeneous ruled minimal hypersurface is a ruled hypersurface, 
which is a Ricci soliton, but not gradient (see \cite{L11}). 

This paper is organized as follows. 
In Section~2, we recall algebraic Ricci solitons, 
which play important roles in studying homogeneous Ricci solitons. 
In Section~3, we recall some necessary results on Lie hypersurfaces in $\CH^n$. 
In Section~4 we prove the main theorem. 

\section{Ricci solitons and algebraic Ricci solitons}

The notion of algebraic Ricci solitons has essentially been 
introduced by Lauret (\cite{L01, L11}). 
In this section, we recall a 
relationship 
between left-invariant Ricci solitons and algebraic Ricci solitons. 
Let $G$ be a simply-connected Lie group, 
and $g$ be a left-invariant metric on $G$. 
The \textit{corresponding metric Lie algebra} of $(G, g)$ 
is a pair $(\fr{g}, \langle, \rangle)$, 
where $\fr{g}$ is the Lie algebra of $G$ 
and $\langle, \rangle$ is the inner product on $\fr{g}$ corresponding to $g$. 
We can study curvatures of $(G, g)$ in term of 
the corresponding metric Lie algebra $(\fr{g}, \langle, \rangle)$. 

First of all, 
we recall some curvature formulas 
for a metric Lie algebra $(\fr{g}, \langle, \rangle)$. 
Let $X, Y \in \fr{g}$.
The Levi-Civita connection $\nabla$ of $(\fr{g}, \langle, \rangle)$ is given by 
\begin{equation}
	2\langle \nabla_X Y, Z \rangle
         = \langle [X, Y], Z \rangle + \langle [Z,X], Y \rangle + \langle X, [Z, Y] \rangle \quad (Z \in \fr{g}).
\end{equation}
Then the \textit{Riemannian curvature} $R$ of $(\fr{g}, \langle, \rangle)$ 
is defined by
\begin{equation}
	R(X, Y) := \nabla_{[X,Y ]} - \nabla_X \nabla_Y + \nabla_Y \nabla_X,
\end{equation}
and the \textit{Ricci operator} $\Ric$ of $(\fr{g}, \langle, \rangle)$ is defined by 
\begin{equation}
	\Ric(X) := \sum R(E_i, X) E_i,
\end{equation}
where $\{E_i\}$ is an orthonormal basis of $(\fr{g}, \langle, \rangle)$. 

We now recall the notion of algebraic Ricci solitons. 
For a Lie algebra $\fr{g}$, 
let us denote the Lie algebra of derivations by 
\begin{equation}
	\Der(\fr{g}) := \{D \in \fr{gl}(\fr{g}) \mid D[X, Y] = [D(X), Y] + [X, D(Y)] \ (\forall X, Y \in \fr{g}) \} . 
\end{equation} 

\begin{Def}
A metric Lie algebra $(\fr{g}, \langle, \rangle)$ is called an \textit{algebraic Ricci soliton}
 if the following holds: 
	\begin{equation}\label{eq_algebraic}
	\Ric = c \cdot \id + D, 
	\end{equation}
for some $c \in \mathbb{R}$ and $D \in \Der{(\fr{g})}$. 
\end{Def}

A 
relationship 
between left-invariant Ricci solitons 
and algebraic Ricci solitons is given as follows. 
According to \cite{L11}, 
if the corresponding metric Lie algebra $(\fr{g}, \langle, \rangle)$ is an algebraic Ricci soliton, 
then $(G, g)$ is a Ricci soliton. 
The converse also holds if $G$ is completely solvable. 
Recall that $G$ is said to be 
\textit{completely solvable} 
if $G$ is solvable and the eigenvalues of any $\ad X$ are all real. 

\begin{Thm}[\cite{L11}] \label{iff}
Let $(G, g)$ be a simply-connected completely solvable Lie group equipped with a left-invariant metric, 
 and $(\fr{g}, \langle, \rangle)$ be the corresponding metric Lie algebra.
Then $(G, g)$ is a Ricci soliton if and only if $(\fr{g}, \langle, \rangle)$ is an algebraic Ricci soliton.
\end{Thm}

\section{Lie hypersurfaces in the complex hyperbolic space}

In this section, 
we recall Lie hypersurfaces in the complex hyperbolic space $\CH^n$ 
and their curvature properties. 
We refer to \cite{B, BT03, HHT}. 

\begin{Def}
The triple $(\fr{s}, \langle, \rangle, J)$ is called the 
\textit{solvable model} 
of the complex hyperbolic space $\CH^n$ if 
	\begin{enumerate}[\normalfont (1)]
	  \item $\fr{s}$ is a Lie algebra, and there exists a basis $\{A_0, X_1, Y_1, \ldots, X_{n-1}, Y_{n-1}, Z_0\}$
          	whose bracket products are given by
                $$[A_0, X_i] = (1/2)X_i, \ 
                  [A_0, Y_i] = (1/2)Y_i, \
                  [A_0, Z_0] = Z_0, \
                  [X_i, Y_i] = Z_0,$$
          \item $\langle, \rangle$ is an inner product on $\fr{s}$ so that the above basis is orthonormal, 
          \item $J$ is a complex structure on $\fr{s}$ given by
          	$$J(A_0) = Z_0, \ J(Z_0) = -A_0, \ J(X_i) = Y_i, \ J(Y_i) = -X_i.$$
        \end{enumerate}
\end{Def}

We use $\fr{s}$ instead of $(\fr{s}, \langle, \rangle, J)$ for simplicity, 
and denote by $S$ the corresponding simply-connected solvable Lie group 
with the induced left-invariant Riemannian metric and complex structure. 
It is known that $S$ can be identified with 
$\CH^n$ with holomorphic sectional curvature $-1$. 
In fact, $S$ coincides with 
the solvable part of the Iwasawa decomposition of $\mathrm{SU}(1,n)$. 

In the rest of the section, 
we recall some known results on Lie hypersurfaces in $\CH^n$. 
First of all, we recall the definition. 

\begin{Def}
A \textit{Lie hypersurface} in $\CH^n$ 
is an orbit of a cohomogeneity one action without singular orbits. 
\end{Def}

It is known that there exist exactly two cohomogeneity one actions 
on $\CH^n$ without singular orbit up to orbit equivalence. 

\begin{Thm}[\cite{BT03}]
An isometric action on $\CH^n$ is 
a cohomogeneity one action without singular orbit if and only if 
it is orbit equivalent to one of the actions of $S(\pi/2)$ or $S(0)$, 
where $S(\theta)$ is the connected Lie subgroup of $S$ with Lie algebra 
$$
\fr{s}(\theta) := \fr{s} \ominus \mathbb{R}(\cos(\theta)X_1 + \sin(\theta)A_0) . 
$$
\end{Thm}

Note that $\ominus$ means the orthogonal complement 
with respect to $\langle , \rangle$. 
By studying the orbits of these actions, 
one obtains the classification
of Lie hypersurfaces in $\CH^n$ up to isometric congruence.

\begin{Thm}[\cite{B}]
Every Lie 
hypersurface
in $\CH^n$ is isometrically congruent to 
the orbit $S(\theta).o$
through the origin $o$ for 
some
$\theta \in [0, \pi/2]$. 
\end{Thm}

Note
that 
$S(\pi/2) = N$ and hence
the orbit $S(\pi/2).o$ is a horosphere. 
The
orbit $S(0).o$ is the homogeneous ruled minimal hypersurface. 
If $0 < \theta < \pi/2$, then
the orbit $S(\theta).o$ is 
isometrically congruent to $S(0).p$ for some $p$, 
which is
an equidistant hypersurface to $S(0).o$.

We now study the geometry of Lie hypersurfaces. 
The Lie hypersurface $S(\theta).o$ can be studied 
in terms of the metric Lie algebra 
$(\fr{s}(\theta) , \langle , \rangle)$, 
where this $\langle , \rangle$ denotes the restriction of 
$\langle , \rangle$ to $\fr{s}(\theta)$. 
First of all, by definition, 
$\fr{s}(\theta)$ is a codimension one subalgebra of $\fr{s}$, 
and it
is solvable. 
Denote 
by
	$$T := \cos(\theta)A_0 - \sin(\theta)X_1, \quad \fr{v}_0 := \span\{X_2, Y_2, \ldots, X_{n-1}, Y_{n-1}\} . $$
Then we have
the orthogonal decomposition
	\begin{equation}
        	\fr{s}(\theta) = \span\{T\} \oplus \span\{Y_1\} \oplus \fr{v}_0 \oplus \span\{Z_0\}.
        \end{equation}
One 
can directly see the following equations, which will be used hereafter.

\begin{Lem} \label{bracket}
Let $V, W \in \fr{v}_0$. Then we have 
	\begin{enumerate}[\normalfont (1)]
	  \item $[T, Y_1] = (1/2)\cos(\theta)Y_1 - \sin(\theta)Z_0$,
	  \item $[T, V] = (1/2)\cos(\theta)V$,
	  \item $[T, Z_0] = \cos(\theta)Z_0$, 
	  \item $[V, W] = \langle [V, W], Z_0 \rangle Z_0$, especially, $[X_k, Y_k] = Z_0$ for all $k =2, \ldots, n-1$.
        \end{enumerate}
\end{Lem}

Finally in this subsection, 
we recall the formulas for Ricci curvatures of the Lie hypersurfaces. %
By direct calculations in terms of
$(\fr{s}(\theta), \langle, \rangle)$, 
one has the following. 

\begin{Prop}[\cite{HHT}] \label{ric}
The Ricci operator of $(\fr{s}(\theta), \langle, \rangle)$ satisfies
	\begin{enumerate}[\normalfont (1)]
	  \item $\Ric(T) = -(1/4)(2+(2n-1)\cos^2(\theta))T, $
	  \item $\Ric(Y_1) = -(1/4)(2+(2n-3)\cos^2(\theta))Y_1 + (n/2)\sin(\theta)\cos(\theta)Z_0,$
	  \item $\Ric(V) = -(1/4)(2+(2n-1)\cos^2(\theta))V$ for any $V \in \fr{v}_0$,
	  \item $\Ric(Z_0) = (n/2)\sin(\theta)\cos(\theta)Y_1 + (1/2)((n-1) - 2n\cos^2(\theta))Z_0.$
        \end{enumerate}
\end{Prop}

\section{Main result}

In this section we 
classify Ricci soliton Lie hypersurfaces in $\CH^n$.
First of all, we see that 
we have only to work on algebraic Ricci solitons.

\begin{Lem}\label{completely solvable}
For any $\theta \in [0,\pi/2]$, the Lie 
hypersurface
$S(\theta).o$ is a Ricci soliton 
if and only if $(\fr{s}(\theta), \langle, \rangle)$ is an algebraic Ricci soliton. 
\end{Lem}

\begin{proof}
Let $\theta \in [0,\pi/2]$.
Owing to Theorem~\ref{iff}, 
we have only to check that $\fr{s}(\theta)$ is completely solvable. 
Take any $X \in \fr{s}(\theta)$. 
One can write
\begin{equation}
X = a_1 T + a_2 Y_1 + V + a_3 Z_0 \quad (\mbox{for $V \in \fr{v}_0$}) . 
\end{equation}
According to Lemma~\ref{bracket}, we can show directly that 
\begin{equation}
	\begin{aligned}
	 \ad X(T) &= -(1/2)a_2\cos(\theta)Y_1 - (1/2)\cos(\theta)V \\ 
& \qquad - (a_2 \sin(\theta) + a_3\cos(\theta)) Z_0,\\
	 \ad X(Y_1) &= (1/2)a_1 \cos(\theta)Y_1 + a_1 \sin(\theta) Z_0,\\
	 \ad X(W) &= (1/2)a_1\cos(\theta)W + \langle [V, W], Z_0 \rangle Z_0,\\
	 \ad X(Z_0) &= a_1 \cos(\theta)Z_0 , 
	\end{aligned}
\end{equation}
where $W \in \fr{v}_0$. 
Hence $\ad X$ is represented by a triangular matrix 
with respect to the canonical basis 
\begin{equation}
\label{eq:basis}
\{ T , Y_1 , X_2 , Y_2 , \ldots , X_{n-1} , Y_{n-1} , Z_0 \} 
\end{equation}
of $\fr{s}(\theta)$. 
Therefore all eigenvalues of $\ad X$ are real, 
which completes
the proof. 
\end{proof}

From now on, we shall study
whether $(\fr{s}(\theta), \langle, \rangle)$ is an algebraic Ricci soliton. 
We separate this into three cases, namely 
$\theta = \pi/2$, $\theta \in ]0, \pi/2[$, and $\theta = 0$, 
and discuss them individually. 

Firstly, 
we see that $(\fr{s}(\pi/2), \langle, \rangle)$ is always an algebraic Ricci soliton.

\begin{Prop}
Let $n \ge 2$. 
Then $(\fr{s}(\pi/2), \langle, \rangle)$ is an algebraic Ricci soliton. 
\end{Prop}

\begin{proof}
We consider the basis 
$\{ X_1, Y_1, \ldots, X_{n-1}, Y_{n-1}, Z_0 \}$ of $\fr{s}(\pi/2)$. 
Denote by $\Ric$
the matrix expression of the Ricci operator of $\fr{s}(\pi/2)$ 
with respect to this basis. 
Then Proposition~\ref{ric} yields that
\begin{equation}
\Ric = \frac{1}{2} \left( \begin{array}{cccc}
			-1	&	&	& 	\\
			    	&\ddots	&       & 	\\
			    	&	&  -1	& 	\\
			    	&	&	& n-1 \\
	  	\end{array} \right) . 
\end{equation}
In order to show that it is an algebraic Ricci soliton, we define 
\begin{equation}
c := - \frac{n+1}{2} , \quad 
D := \frac{n}{2} \left( \begin{array}{cccc}
			1	&	&	& 	\\
			    	&\ddots	&       & 	\\
			    	&	&1	& 	\\
			    	&	&	& 2	\\
	  	\end{array} \right) . 
\end{equation}
Then 
 one 
can show directly that 
$\Ric = c \cdot \id + D$ and $D \in \Der(\fr{s}(\pi/2))$. 
\end{proof} 

\begin{Remark}
Note that it has already 
been
known from \cite{L01} that $(\fr{s}(\pi/2), \langle, \rangle)$ 
is an algebraic Ricci soliton. 
In fact, $\fr{s}(\pi/2)$ is the Heisenberg Lie algebra, 
and $\langle, \rangle$ is the standard inner product. 
\end{Remark}

Secondly, 
we consider the 
case $\theta \in ] 0, \pi/2[$. 
In this case, 
$(\fr{s}(\theta), \langle, \rangle)$ cannot be an algebraic Ricci soliton. 
The proof starts with the following lemma for a derivation of $\fr{s}(\theta)$. 

\begin{Lem}\label{lem1}
Let $n \ge 2$.
If $\theta \in [ 0, \pi/2 [$, 
then $\langle D(Z_0), Y_1 \rangle = 0$ holds for any $D \in \Der(\fr{s}(\theta))$. 
\end{Lem}

\begin{proof}
Take any $D \in \Der(\fr{s}(\theta))$. 
By 
the
definition of 
a
derivation, one has 
	\begin{equation} \label{theta_eq0}
        \langle D[T, Z_0], Y_1 \rangle = \langle [D(T), Z_0], Y_1 \rangle + \langle [T, D(Z_0)], Y_1 \rangle.
        \end{equation}
We calculate the both sides of this equation. 
By
Lemma~\ref{bracket}, 
the left-hand side satisfies 
\begin{equation}\label{theta_eq1}
	\langle D[T, Z_0], Y_1 \rangle
         = \langle D(\cos(\theta)Z_0), Y_1 \rangle
         = \cos(\theta) \langle D(Z_0), Y_1 \rangle.
\end{equation}
Next we calculate the right-hand side. 
It also follows from Lemma~\ref{bracket} that
\begin{equation}
[\fr{s}(\theta) , Z_0] \subset \mathrm{span} \{ Z_0 \} . 
\end{equation}
Then the first term of the right-hand side of (\ref{theta_eq0}) is 
\begin{equation}\label{theta_eq2}
	\langle [D(T), Z_0], Y_1 \rangle = 0.
\end{equation}
In order to calculate the second term, 
take the canonical orthonormal basis 
\begin{equation}
\{ E_i \} = \{ T, Y_1, X_2, Y_2 \ldots, X_{n-1}, Y_{n-1}, Z_0 \} 
\end{equation} 
of $\fr{s}(\theta)$. 
Lemma~\ref{bracket} yields that
\begin{equation}
\langle [T, E_i], Y_1 \rangle = \left\{\begin{array}{ll}
        	(1/2)\cos(\theta) &\quad (E_i = Y_1), \\
                0 &\quad (\text{otherwise}) . 
          \end{array} \right. 
\end{equation} 
Then the second term of the right-hand side of (\ref{theta_eq0}) satisfies 
\begin{equation}\label{theta_eq3}
	\begin{aligned}
	\langle [T, D(Z_0)], Y_1 \rangle
         &= \langle [T, {\textstyle \irotwo \sum} \langle D(Z_0), E_i \rangle E_i], Y_1 \rangle\\
         &= {\textstyle \irotwo \sum} \langle D(Z_0), E_i \rangle \langle [T, E_i], Y_1 \rangle\\
         &= (1/2)\cos(\theta) \langle D(Z_0), Y_1 \rangle.\\
        \end{aligned}
\end{equation}
Altogether, we obtain 
\begin{equation}
\cos(\theta) \langle D(Z_0), Y_1 \rangle 
= 0 + (1/2) \cos(\theta) \langle D(Z_0), Y_1 \rangle . 
\end{equation}
Since $\theta \in [ 0, \pi/2 [$, this completes the proof. 
\end{proof}

\begin{Prop} \label{prop2}
Let $n \ge 2$. 
If $\theta \in ]0, \pi/2[$, 
then $(\fr{s}(\theta), \langle, \rangle)$ is not an algebraic Ricci soliton.
\end{Prop}

\begin{proof}
We show this by contradiction. 
Assume that $(\fr{s}(\theta), \langle, \rangle)$ is an algebraic Ricci soliton. 
By definition, 
there exist $c \in \mathbb{R}$ and $D \in \Der(\fr{s}(\theta))$ such that 
\begin{equation} \label{theta_ric}
	\Ric = c \cdot \id + D . 
\end{equation}
One thus has 
\begin{equation}
\langle \Ric(Z_0), Y_1 \rangle 
= c \langle Z_0, Y_1 \rangle + \langle D(Z_0), Y_1 \rangle . 
\end{equation}
We calculate the both sides of this equation. 
By
Proposition~\ref{ric}, we have 
\begin{equation}
\langle \Ric(Z_0), Y_1 \rangle = (n/2) \sin(\theta)\cos(\theta) . 
\end{equation} 
On the other hand, Lemma~\ref{lem1} yields that
\begin{equation}
c \langle Z_0, Y_1 \rangle + \langle D(Z_0), Y_1 \rangle = 0 . 
\end{equation}
Since $\theta \in ] 0, \pi/2 {\irotwo [}$, this is a contradiction. 
\end{proof}

Lastly, we study the case $\theta = 0$. 
First of all, 
we consider the case $n > 2$, 
and show that $(\fr{s}(0), \langle, \rangle)$ is not an algebraic Ricci soliton. 
As in the 
previous
case, 
we start with the next lemma on a derivation. 
Note that $\fr{v}_0 \neq \{ 0 \}$ if $n>2$. 

\begin{Lem}\label{lem2}
Assume that $n > 2$, and let $D \in \Der(\fr{s}(0))$. 
Then we have 
\begin{enumerate}[\normalfont (1)]
  \item\label{lem2_1} $\langle D(A_0), A_0 \rangle = 0$, and 
  \item\label{lem2_2} $\langle D(X_{\iro k}), X_{\iro k} \rangle 
+ \langle D(Y_{\iro k}), Y_{\iro k} \rangle 
= \langle D(Z_0), Z_0 \rangle$ for ${\iro k} = 2, \ldots, n-1$. 
\end{enumerate}
\end{Lem}

\begin{proof}
Take any $D \in \Der(\fr{s}({\irotwo 0}))$. 
The proof is similar to that of Lemma~\ref{lem1}. 
We use the canonical orthonormal basis 
of $\fr{s}(0)$, 
\begin{equation} 
\{ E_i \} = \{ A_0, Y_1, X_2, Y_2, \ldots, X_{n-1}, Y_{n-1}, Z_0 \} .
\end{equation} 

We show (1). By the definition of 
a
derivation, one has 
\begin{equation} \label{0_eq0}
\langle D[A_0, Z_0], Z_0 \rangle 
= \langle [D(A_0), Z_0], Z_0 \rangle 
+ \langle [A_0, D(Z_0)], Z_0 \rangle . 
\end{equation}
The
left-hand side of (\ref{0_eq0}) 
satisfies
	\begin{equation} \label{0_eq1}
        \langle D[A_0, Z_0], Z_0 \rangle 
= \langle D(Z_0), Z_0 \rangle . 
        \end{equation}
In order to calculate the right-hand side, one needs 
\begin{equation}
\langle [E_i, Z_0], Z_0 \rangle = \left\{\begin{array}{ll}
        	1 &\quad (E_i = A_0), \\
                0 &\quad (\text{otherwise}) . 
          \end{array}\right . 
\end{equation}
Hence we have 
	\begin{equation} \label{0_eq2}
	\langle [D(A_0), Z_0], Z_0 \rangle
         = \textstyle \sum \langle D(A_0), E_i \rangle 
\langle [E_i, Z_0], Z_0 \rangle
         = \langle D(A_0), A_0 \rangle.
        \end{equation}
Similarly, one can see that 
\begin{equation}
\langle [A_0, E_i], Z_0 \rangle = \left\{\begin{array}{ll}
        	1 &\quad (E_i = Z_0) , \\
                0 &\quad (\text{otherwise}) . 
          \end{array}\right. 
\end{equation}
This yields that
	\begin{equation} \label{0_eq3}
	\langle [T, D(Z_0)], Z_0 \rangle
         = {\textstyle \irotwo \sum} \langle D(Z_0), E_i \rangle \langle [T, E_i], Z_0 \rangle
         = \langle D(Z_0), Z_0 \rangle.
        \end{equation}
Altogether, we have 
\begin{equation}
\langle D(Z_0), Z_0 \rangle 
= \langle D(A_0), A_0 \rangle + \langle D(Z_0), Z_0 \rangle . 
\end{equation}
This completes the proof of the first assertion. 

We show (2). 
Take any $k = 2, \ldots , n-1$. 
We start from 
\begin{equation}
\langle D[X_{\irotwo k}, Y_{\irotwo k}], Z_0 \rangle 
= \langle [D(X_{\irotwo k}), Y_k ], Z_0 \rangle 
+ \langle [X_k, D(Y_k)] , Z_0 \rangle . 
\end{equation}
One can calculate each term as follows: 
\begin{equation}
\begin{aligned}
\langle D[X_{\irotwo k},   Y_{\irotwo k}], Z_0 \rangle &= \langle D(Z_0), Z_0 \rangle, \\
\langle [D(X_{\irotwo k}), Y_{\irotwo k}], Z_0 \rangle 
&= 
\textstyle {\irotwo \sum} \langle D(X_{\irotwo k}), E_{\irotwo i} \rangle \langle [E_{\irotwo i}, Y_{\irotwo k}], Z_0 \rangle 
= \langle D(X_{\irotwo k}), X_{\irotwo k} \rangle , \\
\langle [X_{\irotwo k}, D(Y_{\irotwo k})], Z_0 \rangle 
&= 
\textstyle {\irotwo \sum} \langle D(Y_{\irotwo k}), E_{\irotwo i} \rangle \langle [X_{\irotwo k}, E_{\irotwo i}], Z_0 \rangle 
= \langle D(Y_{\irotwo k}), Y_{\irotwo k} \rangle.
\end{aligned}
\end{equation}
This completes the proof of the second assertion. 
\end{proof}

\begin{Prop}
Let $n > 2$. 
Then $(\fr{s}(0), \langle, \rangle)$ is not an algebraic Ricci soliton. 
\end{Prop}

\begin{proof}
We show this by contradiction. 
Assume that $(\fr{s}(0), \langle, \rangle)$ is an algebraic Ricci soliton. 
By definition,
there exist $c \in \mathbb{R}$ and $D \in \Der(\fr{s}(0))$ such that 
\begin{equation} 
\Ric = c \cdot \id + D . 
\end{equation}
Then, for any $E \in \fr{s}(0)$ with $\langle E, E \rangle = 1$, one has 
	\begin{equation} \label{0_c}
        \langle \Ric(E), E \rangle = \langle cE + D(E), E \rangle = c + \langle D(E), E \rangle . 
        \end{equation}
 Let $V \in \fr{v}_0$ be a unit vector, 
and we substitute $A_0$, $V$, and $Z_0$ for $E$. 
Then, Proposition~\ref{ric} yields that
\begin{equation} \label{1_c}
\begin{aligned}
c + \langle D(A_0), A_0 \rangle 
& = \langle \Ric(A_0), A_0 \rangle = - (2n+1) / 4 , \\ 
c + \langle D(V), V \rangle 
& = \langle \Ric(V), V \rangle = - (2n+1) / 4 , \\ 
c + \langle D(Z_0), Z_0 \rangle 
& = \langle \Ric(Z_0), Z_0 \rangle = - (n+1) / 2 . 
\end{aligned}
\end{equation}
These
equations yield that 
\begin{equation}\label{1_d}
	\langle D(A_0), A_0 \rangle 
= \langle D(V), V \rangle \ne \langle D(Z_0), Z_0 \rangle.
\end{equation}
Lemma~\ref{lem2} (\ref{lem2_1}) 
yields
that 
\begin{equation}
0 = 
\langle D(A_0), A_0 \rangle = 
\langle D(V), V \rangle.
\end{equation}
Since $n>2$, we have $X_2, Y_2 \in \fr{v}_0$, 
which can be substituted for $V$. 
Thus, owing to Lemma \ref{lem2} (\ref{lem2_2}), we have 
	\begin{equation}
        \langle D(Z_0), Z_0 \rangle = \langle D(X_2), X_2 \rangle + \langle D(Y_2), Y_2 \rangle = 0. 
        \end{equation}
This is a contradiction. 
\end{proof}

If $n = 2$, however, $(\fr{s}(0), \langle, \rangle)$ is indeed an algebraic Ricci soliton.

\begin{Prop}
Let $n = 2$. 
Then $(\fr{s}(0), \langle, \rangle)$ is an algebraic Ricci soliton. 
\end{Prop}

\begin{proof}
Denote by $\Ric$ the matrix expression of the Ricci operator 
with respect to the basis $\{ A_0, Y_1, Z_0 \}$ of $\fr{s}(0)$. 
Then, 
Proposition~\ref{ric}
yields that
\begin{equation}
\Ric
          = \left( \begin{array}{ccc}
			-5/4 &      &      \\
			     & -3/4 &      \\
			     &      & -3/2 \\
	  	\end{array} \right) . 
\end{equation} 
Let us define
\begin{equation}
c := - \frac{5}{4} , \quad 
D := \left( \begin{array}{ccc}
			0    &     &      \\
			     & 1/2 &      \\
			     &     & -1/4 \\
	  	\end{array} \right) . 
\end{equation} 
Then we can show directly that 
$\Ric = c \cdot \id + D$, and $D \in \Der(\fr{s}(0))$. 
\end{proof} 

\begin{Remark}
It is known that the Lie hypersurface 
$(\fr{s}(0), \langle, \rangle)$ is an algebraic Ricci soliton when $n=2$. 
In fact, Lauret (\cite{L11}) has classified three-dimensional 
algebraic Ricci solitons on solvable Lie algebras. 
Among them, there are Lie algebras 
$\fr{r}_\alpha = \span\{A, X, Y\}$, for $\alpha \in [-1, 1]$, 
whose bracket relations are given by 
\begin{equation}
[A, X] = X, \quad [A, Y] = \alpha Y . 
\end{equation}
Let $\langle, \rangle_0$ be 
the inner product so that the basis $\{A, X, Y\}$ is orthonormal. 
Then $(\fr{r}_\alpha, \langle, \rangle_0)$ is an algebraic Ricci soliton 
(\cite{L11}). 
Note that our Lie hypersurface $(\fr{s}(0), \langle, \rangle)$ is 
obviously isometric to $(\fr{r}_{1/2}, \langle, \rangle_0)$. 
\end{Remark}

In consequence of the arguments above we conclude the following theorem. 

\begin{Thm}
The Lie hypersurfaces $S(\theta).o$ in $\CH^n$ is a Ricci soliton 
if and only if 
\begin{enumerate}
\item[$(1)$]
$\theta = \pi/2$, or 
\item[$(2)$]
$n=2$ and $\theta = 0$. 
\end{enumerate}
\end{Thm}


\end{document}